\documentclass[article, 12pt,reqno]{amsart}

\usepackage{amsmath,amssymb,amsthm,graphicx}
\usepackage{latexsym}
\usepackage{amscd}
\usepackage{amsfonts}
\usepackage[usenames,dvipsnames]{color}
\usepackage{verbatim}
\usepackage{float}
\usepackage{url}
\usepackage{xypic}
\usepackage{marginnote}\usepackage{mathtools}
\usepackage[pdftex,colorlinks]{hyperref}
\usepackage[mathscr]{euscript}
\usepackage[T1]{fontenc}

\usepackage{fullpage}

\usepackage{enumitem}

\usepackage{lineno}

\usepackage{soul} 

\usepackage{hyperref}
\hypersetup{
	colorlinks=true,
	linkcolor=blue,
	citecolor=blue,
	urlcolor=blue,
	pdfauthor={DC & FGG},
	pdftitle={}
}

\newtheorem{thmalpha}{Theorem}  
\newtheorem{coralpha}[thmalpha]{Corollary}  

\newtheorem{thm}{Theorem}[section] 		
\newtheorem{theorem}[thm]{Theorem}		
\newtheorem{prop}[thm]{Proposition}		
\newtheorem{lem}[thm]{Lemma}			
\newtheorem{cor}[thm]{Corollary}		

\newtheorem*{question*}{Question}

\theoremstyle{definition}	
\newtheorem{remark}[thm]{Remark}			
\newtheorem{definition}[thm]{Definition}	
\newtheorem{rem}[thm]{Remark}				
\newtheorem{example}[thm]{Example}			

\newtheorem*{rem*}{Remark}			
\newtheorem*{notation*}{Notation}	
\newtheorem*{ack}{Acknowledgements}	

\numberwithin{equation}{section}

\newenvironment{eqnenumerate}
{\enumerate[
	label=(\thesection.\arabic*),
	before=\setcounter{enumi}{\value{equation}},
	after=\setcounter{equation}{\value{enumi}},
	]}
{\endenumerate}

\DeclareMathOperator{\Int}{Int}

\DeclareMathOperator{\eucl}{eucl}

\DeclareMathOperator{\length}{Length}

\DeclareMathOperator{\id}{id}

\DeclareMathOperator{\Sym}{Sym}

\DeclareMathOperator{\inte}{int}

\DeclareMathOperator{\sgn}{sgn}

\newcommand{\bR}{\mathbb{R}}
\newcommand{\ra}{\rightarrow}

\newcommand{\bN}{\mathbb{N}}

\newcommand{\KL}{K\L }
\newcommand{\R}{\mathbb{R}}

\newcommand{\interior}[1]{%
	{\kern0pt#1}^{\mathrm{o}}%
}

\makeatletter
\newsavebox\myboxA
\newsavebox\myboxB
\newlength\mylenA

\newcommand*\xoverline[2][0.75]{%
	\sbox{\myboxA}{$\m@th#2$}%
	\setbox\myboxB\null
	\ht\myboxB=\ht\myboxA%
	\dp\myboxB=\dp\myboxA%
	\wd\myboxB=#1\wd\myboxA
	\sbox\myboxB{$\m@th\overline{\copy\myboxB}$}
	\setlength\mylenA{\the\wd\myboxA}
	\addtolength\mylenA{-\the\wd\myboxB}%
	\ifdim\wd\myboxB<\wd\myboxA%
	\rlap{\hskip 0.5\mylenA\usebox\myboxB}{\usebox\myboxA}%
	\else
	\hskip -0.5\mylenA\rlap{\usebox\myboxA}{\hskip 0.5\mylenA\usebox\myboxB}%
	\fi}
\makeatother

\usepackage{todonotes}

\newcommand{\pref}[2]{\hyperref[#2]{#1 \ref*{#2}}}

\begin{document}
\title[{\KL} Functions and Mapping Cylinder Neighborhoods]{ Kurdyka--\L ojasiewicz functions and Mapping Cylinder Neighborhoods}
\date{\today}

\author[D.~Cibotaru]{Daniel Cibotaru}
	\address[D.~Cibotaru]{Departamento de Matem\'atica, Universidade Federal do Cear\'a, Fortaleza, Brazil.}
	\email{\href{mailto:daniel@mat.ufc.br}
		{daniel@mat.ufc.br}}

\author[F.~Galaz-Garc\'ia]{Fernando Galaz-Garc\'ia}
\address[F.~Galaz-Garc\'ia]{Department of Mathematical Sciences, Durham University, United Kingdom.}
\email{\href{mailto:fernando.galaz-garcia@durham.ac.uk}
	{fernando.galaz-garcia@durham.ac.uk}}



\begin{abstract}
	Kurdyka--\L ojasiewicz (\KL) functions are real-valued functions characterized by a differential inequality involving the norm of their gradient. This class of functions is quite rich, containing objects as diverse as subanalytic, transnormal or Morse functions. 
	We  prove that the zero locus of a Kurdyka--\L ojasiewicz function admits a mapping cylinder neighborhood. This implies, in particular, that  wildly embedded topological $2$-manifolds in $3$-dimensional Euclidean space, such as Alexander horned spheres, do not arise as the zero loci of \KL{} functions.
	
\end{abstract}
\maketitle

\setcounter{tocdepth}{1}
\tableofcontents



\section{Introduction}

A basic result of  differential topology says that any closed subset of a smooth manifold $M$, no matter how complicated, can arise as the zero locus of a smooth, nonnegative real-valued function on $M$.
On the other hand, if $M$ is analytic, the zero loci of  real analytic functions on $M$  form, of course, a more restricted class of closed sets whose properties have been intensively studied for more than a century. Surprisingly, topological restrictions on these sets come from a differential inequality: 
\L ojasiewicz  proved in \cite[Proposition 1, p.\ 92]{Lo1} that  a real analytic  function $f\colon \R^n\to \R$  satisfies, in a neighborhood of  a point $p\in f^{-1}(0)$,  the gradient inequality
\begin{equation}\label{EQ:Lo1} |\nabla_x f|\geq  C|f(x)|^{\theta}
\end{equation}
for some $C>0$ and $0<\theta<1$.

The class of functions  that satisfy  gradient inequalities similar to (\ref{EQ:Lo1})  is quite larger than that of real analytic functions on $\R^n$. Kurdyka provided a pivotal extension in \cite{K}, where he showed that real-valued functions whose graph belongs to an o-minimal structure $\mathcal{M}$ on $(\mathbb{R},+,\cdot)$ satisfy,  for some  constants $C,\rho>0$, the inequality
\begin{equation}\label{eq02} |\nabla_x (\psi \circ f)|\geq C\text{ for all }~x\in f^{-1}(0,\rho).
\end{equation}

In \eqref{eq02}, the function $\psi\colon [0,\rho)\ra[0,\infty)$ is strictly increasing, continuous, and $C^1$ on $(0,\rho)$,  with $\psi(0)=0$. In fact, $\psi$ can be chosen to belong to $\mathcal{M}$. Examples  of functions satisfying a gradient inequality such as (\ref{eq02}) can easily be  found beyond the world of  o-minimal structures. It is enough to recall that
the  distance function  to a closed, smooth submanifold of a Riemannian ambient manifold (see, for example,  \cite{Pet})  has a gradient of norm $1$ in a neighborhood of the submanifold. It is therefore unsurprising that functions satisfying such inequalities, henceforth called \emph{{\KL} functions}, have become the focus of stimulating research in the past two decades. This was motivated also by the interest coming from convex optimization,   complexity theory, and neural networks.  The article by Bolte, Daniilidis,  Ley, and Mazet \cite{BDLM},  a benchmark on the topic of K\L{} functions,  and the references therein, will  give the reader a broader view of their alternative characterizations and  applications.  The recently found counterexample to the Thom Gradient Conjecture \cite{DHL} for the class of K\L{}  functions  only  increases the intrigue surrounding them.


Our original interest in  K\L{} functions stemmed from the topological properties of  their zero locus.  In \cite{K},  Kurdyka  proved  that the zero locus of a K\L{} function is a strong deformation retract of a neighborhood using the natural deformation induced by the negative gradient flow. Having a neighborhood which is a strong deformation retract is the first step in proving stronger ``embedding'' properties of the zero locus, like, for example, that the pair $(M,f^{-1}(0))$ is a cofibration (see \cite{AS}). However, the cofibration property alone is still not enough to eliminate what one would perceive as pathological or wild embeddings of topologically nice spaces such as spheres. We therefore prove something stronger.


\begin{thmalpha}
	\label{T:MAIN_THM}
	The zero locus of a Kurdyka--\L ojasiewicz function has a regular mapping cylinder neighborhood. Moreover, this neighborhood is forward-time
	 invariant with respect to the negative gradient flow.
\end{thmalpha}

A mapping cylinder neighborhood is, roughly speaking, a topological generalization for a closed subset of the notion of tubular neighborhood for a smooth submanifold; for a precise definition, see Definition \ref{DEF:MCN} in Subsection \ref{MCNsbs}. In this sense, one may think of Theorem~\ref{T:MAIN_THM} as an analog of the tubular neighborhood theorem for smooth submanifolds. 
  
 The question of which closed subsets $Z$ of a topological space $M$ have a mapping cylinder neighborhood  was intensively explored by topologists in the 1970s. In a celebrated result, obtained as an application of the theory of maps completion  (see  \cite{Qu1,Qu}), Quinn showed that if $Z$ is an absolute neighborhood retract (ANR) whose complement is locally $1$-connected ($1$-LC), and $M$ is a topological manifold of dimension at least $5$, then $Z$ has a mapping cylinder neighborhood.  The context of Quinn's results  and  of our note as well  is that of relative manifolds $(M,Z)$. This means that $M\setminus Z$ is a manifold (topological in \cite{Qu1}, of class at least $C^2$ for us).  In contrast to \cite{Qu1}, we do not assume that either $M$ or $Z$ are ANRs, nor that $M\setminus Z$ is $1$-LC. Thus, even in dimension $\geq 5$, Theorem \ref{T:MAIN_THM} does not seem to follow directly from Quinn's general results.
   
In dimension $3$, Moise \cite{Mo} proved that every topological $3$-manifold has a piecewise linear (PL) structure unique up to PL-isomorphism. More generally, every topological $3$-manifold has a differentiable structure unique up to diffeomorphism (see \cite[Theorem 2]{Mi2}). Thus, every topological $3$-manifold  can be triangulated.
In an article based on his Ph.D.\ thesis \cite{Ni},   Nicholson  proved that, given a  $3$-manifold $M$  and a closed subset $C$ which is a topological complex (i.e., as a topological space, $C$ is homeomorphic to some locally finite simplicial complex), then $C$ is \emph{tamely embedded} in $M$ (i.e., there exists a triangulation of $M$ such that $C$ is a subcomplex)  if and only if $C$ has a mapping cylinder neighborhood. If $C$ is not tamely embedded,  one says that its embedding is \emph{wild} and $C$ is \emph{wildly embedded}.  

An embedding  of the $2$-sphere $\mathbb{S}^2$ into the $3$-sphere $\mathbb{S}^3$ is \emph{flat} if it is topologically equivalent to the embedding of the equatorial sphere, i.e., there exists an (ambient) homeomorphism $\mathbb{S}^3\ra \mathbb{S}^3$ that turns the obvious diagram with the two embeddings of $\mathbb{S}^2$ commutative.   For example, an embedding of $\mathbb{S}^2$ into $\mathbb{S}^3$ with non simply connected complement clearly cannot be flat. By an old result of Alexander \cite{Al2} (see also \cite{Br}),  an embedding  of $\mathbb{S}^2$ into $\mathbb{S}^3$ is flat if and only if it is  tame.  Therefore, any wild embedding of $\mathbb{S}^2$ into $\mathbb{S}^3$ is not flat. 

One notorious example of a wild embedding is that of the \emph{Alexander horned sphere} \cite{Al}, an embedding of $\mathbb{S}^2$ into $\mathbb{S}^3$ such that at least one of the two connected components of the complement is not simply connected. The existence of such embeddings was a consequential result in topology, first exhibited by Alexander \cite{Al} almost a hundred years ago as a counterexample to the Schoenflies Theorem in dimensions at least $3$. For further examples of wild $2$-spheres, such as the \emph{Antoine, Fox--Artin,} and \emph{Bing  spheres}, 
we refer the reader to \cite{BC}. 
An uncountable family of nonequivalent
Alexander horned sphere embeddings  in the Sobolev class $W^{1,n}$ was produced more recently by Haj{\l}asz and Zhou in \cite{HZ}.

Any compact, topological $2$-manifold is homeomorphic to a finite simplicial complex and hence a wild embedding of such a manifold  cannot have a mapping cylinder neighborhood by Nicholson's theorem.  We thus conclude from our Theorem \ref{T:MAIN_THM}  the following result.


\begin{coralpha}
No wildly embedded topological $2$-manifold in a $3$-manifold can be the zero locus of a Kurdyka--\L ojasiewicz function. 
\end{coralpha}

The second (and last) part of our note aims to complement the results of \cite{BDLM} in terms of alternative  characterizations of a  \KL{}  function in a neighborhood of a noninterior  point  of its zero locus $Z$. Given the general context of the results of \cite{BDLM}, some type of semiconvexity condition is necessary there  for a good development of the theory of subdifferential evolution equations. As opposed to \cite{BDLM},   we do not assume  any property resembling  semiconvexity for the functions we consider.  In counterpart, they  are more regular on $M\setminus Z$ and our context is  finite dimensional, with local compactness of $M$ playing a decisive role.

Given a relative manifold $(M,Z)$ and a nonnegative, continuous function $f\colon M\ra [0,\infty)$ with $Z=f^{-1}(0)$ and $f$ of class $C^1$ on $M\setminus Z$, we will say that a point $p\in \partial Z$ is \emph{simple nondegenerate} if it has a neighborhood $U_p$ with $\nabla_x f\neq 0$ for all $x\in U_p\setminus Z$.
If the inequality \eqref{eq02} holds locally around $p$, we say that $p$ is \emph{K{\L} nondegenerate} and, if the desingularization function $\psi$ in \eqref{eq02} can be chosen to be absolutely continuous, rather than $C^1$, we will say that $p\in \partial Z$ is \emph{weakly nondegenerate} (see Definitions~\ref{DEF:KL_FUNCTION} and \ref{DEF:SIM_WKL NONDEG_PT} for more details). We obtain the following characterization.

\begin{thmalpha}\label{T:MAIN_THM2} Let $f\colon M\ra [0,\infty)$ be a continuous function, with zero locus $Z$ and of class $C^1$ on $M\setminus Z$. If $p\in \partial Z$ is a simple nondegenerate point, then the following assertions are equivalent:
	\begin{enumerate}
		\item The point $p\in \partial Z$ is  K{\L} nondegenerate,\label{T:MAIN_THM2_A}
		\item The point $p\in \partial Z$ is weakly K{\L} nondegenerate,\label{T:MAIN_THM2_B}
		\item There exists a compact neighborhood $K_p\ni p$ such that the  upper semicontinuous  function 
		\[
		\alpha^K(t)=\displaystyle{\frac{1}{\displaystyle\inf_{x\in f^{-1}(t)\cap K_p}|\nabla_x f|}}
		\]
		is in $L^1(0, \rho)$ for some $\rho>0$. \label{T:MAIN_THM2_C}
		\item There exists an open  neighborhood $U\ni p$, a positive number $\rho>0$,  and a continuous function $a\colon (0,\rho]\ra (0,\infty)$ such that  $a^{-1}\in L^1(0,\rho)$, and
		\begin{align}\label{eq27} |\nabla_xf|\geq a(f(x)) \text{ for all } x\in U\setminus Z.
		\end{align}
		\label{T:MAIN_THM2_D}
	\end{enumerate}
\end{thmalpha}

 The equivalence of K{\L} nondegeneracy with the integrability condition \eqref{T:MAIN_THM2_C} in Theorem~\ref{T:MAIN_THM2} above appears also in \cite[Theorems 18 and 20]{BDLM}.   In a sense, our proof of Theorem \ref{T:MAIN_THM2}, which does not use the (sub)gradient curves, is a shortcut adapted to the context of relative, finite dimensional manifolds. One important lemma from \cite{BDLM} about integrable majorants  of upper semicontinuous functions appears prominently in our proof as well.

We conclude this note with a brief study of 
what one can describe as the  ``opposite'' gradient inequality to \eqref{eq27}, namely
 \begin{align}
 \label{eq03} |\nabla f|\leq b(f),
 \end{align}
 where $b$  is a nonnegative continuous function with $\int_0^\rho b^{-1}=\infty$ for some $\rho>0$. We note that such an equation holds in a neighborhood of a simple non-degenerate point $p\in \partial Z$ if and only if
 \[ 
 \int_0^{\rho}\frac{1}{\displaystyle\sup_{x\in f^{-1}(t)\cap K_p} |\nabla_x f|}~dt=\infty
 \]
 for some $\rho> 0$ and a compact neighborhood $K_p$ of $p$.
 
   A quick look at the one-dimensional case shows that there do not exist non-negative $C^1$ functions that vanish at a point on the real line and satisfy (\ref{eq03}).  We prove that  no point in the boundary of $Z$ can be reached from $M\setminus Z$ by a rectifiable curve.   An alternative characterization of a gradient inequality of type (\ref{eq03}) is also presented making use of a ``dual'' version  of Lemma 45 from \cite{BDLM}. This leads to a classification of simple non-degenerate critical points of non-negative functions in three distinct classes for which we provide examples.

Some major advances in the topology of smooth manifolds have been pushed by the idea that critical points of smooth functions say something about the change in topology. The following question, whose answer we believe to be negative, is a timid attempt at an ``inverse problem'': \emph{The topology of the zero locus might say something about the immediate presence of critical points}.  
	
	
	\begin{question*}  Can a wild 2-sphere be the zero locus of a nonnegative $C^2$ function $f\colon S^3\ra [0,\infty)$ with only simple nondegenerate points?	\end{question*}

Our article is organized as follows. In Section~\ref{S:PRELIM} we collect basic definitions and examples, and set up the framework for the proof of Theorem~\ref{T:MAIN_THM}. We prove this result in Section~\ref{S:EXISTENCE_MCN}.
 Section~\ref{Sec5} contains the proof of Theorem~\ref{T:MAIN_THM2}, along with further considerations on nondegenerate points of \KL\ functions.


\begin{ack}
The first named author would like to thank Vincent Grandjean, Aris Daniilidis, Luciano Mari and Edson Sampaio for some useful discussions. The authors would also like to thank the referee for helpful suggestions.
\end{ack}


\section{Preliminaries}
\label{S:PRELIM}

Let $(M,d)$ be a connected, locally compact metric space and let $Z\subset M$ be a nonempty, closed, proper subspace such that $M\setminus Z$ is a $C^k$ ($k\geq 1$) manifold of dimension $n\geq 1$. Following Spanier \cite[Ch.~6, Sec.~ 2]{Sp}, we call the pair $(M,Z)$ a \emph{relative $C^k$ manifold}.

We emphasize that $(M,d)$ need not be complete. We will require, however, that the restriction of the metric $d$ to $M\setminus Z$ comes from a  Riemannian metric on this manifold. 

The set $Z$ might have nonempty interior.  Since $M$ is connected and $Z\subset M$ is proper, closed, and nonempty, $\partial Z\neq \emptyset$. This ensures that Definition \ref{DEF:KL_FUNCTION} is not vacuous.

 We will call a function  $\alpha\colon (A_1,B_1)\ra (A_2,B_2)$ between pairs of sets, i.e., $B_i\subset A_i$,    a \emph{relative} function if both $\alpha(B_1)\subset B_2$ and $\alpha(A_1\setminus B_1)\subset A_2\setminus B_2$ hold. This implies that $\alpha^{-1}(B_2)=B_1$ and $\alpha^{-1}(A_2\setminus B_2)=A_1\setminus B_1$.   
 
 The $C^l$ morphisms in the category of relative $C^k$-manifolds with $l\leq k$ considered in this article will be  continuous  relative functions $f\colon (M_1,Z_1)\ra (M_2,Z_2)$ such that $f\bigr|_{M_1\setminus Z_1}$ is of class $C^l$ .



	
\begin{definition}
\label{DEF:KL_FUNCTION}	
Let $(M,Z)$ be a relative $C^1$ manifold and let $f\colon (M,Z)\ra ([0,\infty),\{0\})$ be a relative $C^1$ function. 
A point  $p\in \partial Z$  is  \emph{\KL-nondegenerate} (with respect to $f$)  if there exists a triple $(\rho,U,\Psi)$ with $\rho>0$ a positive  constant, $U\subset M$ an open neighborhood of $p$, and $\Psi\colon ([0,\rho),\{0\})\ra ([0,\infty),\{0\})$  a $C^1$ relative function satisfying the following conditions:\vspace{5pt}
\begin{itemize}[leftmargin=1.7cm]
\item[(\KL1)] $\Psi'(t)>0$, for $t\in (0,\rho)$;\vspace{5pt}
\item[(\KL2)] $|\nabla_x(\Psi\circ f)|\geq 1$ for all $x\in f^{-1}((0,\rho))\cap U$. \label{DEF:KL2}
\end{itemize}
\vspace{5pt}

The relative function  $f$ is a  \emph{Kurdyka--\L ojasiewicz function} (or, for short, a \emph{{\KL} function})  if every point $p\in \partial Z$ is \KL-nondegenerate.

The function $f$ is a \emph{global} (or \emph{uniform}) \textit{{\KL} function }if the triple $(\rho,U,\Psi)$ can be chosen so that $U$ is a neighborhood of $\partial Z$. 

A relative $C^1$ function $f\colon(M,Z)\ra (\bR,\{0\})$, not necessarily nonnegative, is  a K{\L} function if its absolute value $|f|$ is K{\L}.
\end{definition}


\begin{example}
\label{examples}
\begin{enumerate}

\item
	
	Let $M_o\subset \bR^n$ be a connected,  embedded $C^k$-submanifold of dimension $m<n$ with nonempty (topological) boundary $\partial M_o$. Clearly, $M_o$ and $\partial M_o$ are disjoint. Let $Z\subset \partial M_o$ be an open subset with respect to the 
	subspace topology on $\partial M_o$. Let $M=M_o\cup Z$ and observe that  $Z$ is  closed in $M$. Moreover,  $Z=U\cap \partial M_o$, where $U$ is open in $\bR^n$.  The set $U'=U\cap (M_o\cup \partial M_o)$ is  locally compact, as it is the intersection of an open and a closed subset. Since $U'=U\cap M$ is open in $M$, and $Z\subset U'$, it follows that every point  $p\in Z$ has a neighborhood in $M$ which is locally compact. This property is obviously  shared also by the points  $p\in M_o$.  It follows that $M$ is locally compact and $(M,Z)$ is a relative manifold.
	
	Observe that, if we let $Z\subset \partial M_o$ be closed in $\partial M_o$ rather than open, then $M_o\cup Z$ might not be locally compact. For example, when  $M_o$ is the open upper half plane in $\R^2$ and $Z$ is a compact segment on the $x$-axis.

\item Let $U$ be an open and bounded subset of $\R^n$. By the Kurdyka--\L{ojasiewicz} Theorem (\cite[Theorem 1]{K}), if $f\colon U\ra \bR$ is a positive differentiable tame function  (with respect to some o-minimal structure on $(\R,+,\cdot)$ containing $U$), then there exists $\Psi$ as in Definition~\ref{DEF:KL_FUNCTION} (cf.\ \cite[Introduction]{DHL}). In particular, real analytic and subanalytic functions are K{\L}.

\item If $\psi\colon ([0,\infty),\{0\})\ra ([0,\infty),\{0\})$ is a relative $C^1$ function with $\psi'>0$, then $\psi\circ f$ is a K{\L} function whenever $f$ is a K{\L} function.

\item Let $M$ be a Riemannian manifold (possibly with boundary) and let $Z\subset M$ be a smooth, properly embedded submanifold  of $M$. Fix $p>0$ and, given $x\in M$, let $f(x)=\tilde{d}(x,Z)^p$, where $d(\cdot,Z)$ is the distance function to $Z$ and $\tilde{d}(\cdot, Z)$ is a function that coincides with $d$ on an open neighborhood $U$ of $Z$ and is smooth and nonvanishing away from $Z$.  One  takes $\Psi(t)=t^{1/p}$. It is well-known that $|\nabla d(\cdot, Z)|\equiv 1$  on some $U\setminus Z$. Clearly, if $p=1$, $f$ is just continuous at $Z$.

\item Nonnegative transnormal functions (see Wang \cite{Wa}) are K\L. These are functions $f\colon M\ra \bR$ on a smooth manifold that satisfy $|df|^2=b(f)$ for some smooth function $b$ (see  Remark~\ref{remtrans}).

\item Let $f\colon (M,Z)\ra ([0,\infty),\{0\})$ be a relative $C^1$ function. Observe that a necessary condition for $f$ to be {\KL}  is that there should exist a neighborhood $U$ of $Z$ such that $\nabla f$ does not vanish on $U\setminus Z$.  An example of a function $f$ on $\bR^2$ with  a gradient orbit spiraling around  the unit circle $S^1\subset \R^2$ which is  a connected component of the zero locus $Z$ of $f$ may be found in \cite[p.~14]{PdM}.  The gradient of $f$ does not vanish  on some $U\setminus S^1$ and $f$ cannot be K{\L}, as the gradient orbit will have infinite length. The zero locus $Z$ is disconnected for this $f$, with the two other connected components of $Z$ spiraling around $S^1$. In \cite{BDLM}, the authors gave an example of a nonnegative $C^2$ convex function on $\bR^2$ whose minimum locus is $S^1$ and which is not K{\L}. More recently, Bolte and Pauwels \cite{BoltePauwels} constructed convex  non-\KL{} functions with regularity $C^k$ $(k\geq 1)$.

\item Let $f\colon M\ra \bR$ be a Morse function on a smooth  $n$-dimensional Riemannian manifold $M$ and suppose that $0$ is a critical value. Then $f^+=\max\{f,0\}$ is a K{\L} function with zero locus $Z=f^{-1}((-\infty,0])$. Note that $f^+$ is obviously K{\L} at a point $p\in \partial Z=f^{-1}(0)$ which is noncritical for $f$.  If $p\in \partial Z$ is a critical point of $f$, one may use the Morse lemma \cite{Mi1}  to produce a coordinate chart $\varphi\colon V\subset M \to U\subset \bR^n$ with $p\in V$ and such that $\varphi(p) = 0\in U$. Letting $x=(x_1,\ldots,x_n)\in \R^n$, the local expression for $f$ on $U$ is given by
\[ 
f(x)=\frac{1}{2} (x_1^2+\ldots +x_k^2-x_{k+1}^2-\ldots x_n^2).
\]
In these coordinates, the metric  on $M$ is given by a smooth function $g\colon U\ra \Sym^+(\bR^n)$  into the positive symmetric matrices $\Sym^+(\bR^n)$ with $g(0)=I_n$, the $n\times n$ identity  matrix.
The gradient is then
\begin{align*}
	\nabla_{x}f & = g^{-1}(x_1,\ldots,x_k,-x_{k+1},\ldots,-x_n)\\
				& =g^{-1}(Jx),
\end{align*}
where $J=\left(\begin{array}{cc} I_{k}& 0\\
 0& -I_{n-k}\end{array}\right)$. 
Hence, if we let $G\colon U\ra \Sym^+(\R^n)$ be given by
\[
G(x)=(g^{-1}J)^T(g^{-1}J),
\]
 then
\begin{equation}\label{nbG} |\nabla_{x} f|^2=\|x\|_{G(x)}^2,
\end{equation}
where $\|\cdot\|_{G(x)}$ is the norm determined by $G(x)$. 
Since $G$ is continuous and $G(0)=I_n$, there exists $W\subset U$ and $0<C<1$ such that $\lambda_{G(x)}\geq C$ for all $\lambda\in \sigma(G(x))$, for all $x\in W$. It follows then that
\begin{equation}\label{nbG1}\|x\|_{G(x)}^2\geq C\|x\|_{\eucl}^2 \text{ for all } ~x\in W.
\end{equation}

We obviously have $\|x\|_{\eucl}^2\geq 2|f(x)|$, which together with (\ref{nbG}), (\ref{nbG1}) implies that
\[
|\nabla_{x}f^+|^2\geq 2C f^+(x)
\]
for all $x\in W\setminus (f^+)^{-1}(0)$. So $\Psi(t)=C't^{1/2}$ for some $C'>0$. A similar reasoning works if $f$ is a Morse--Bott function. \end{enumerate}
\end{example}

\subsection{Mapping cylinder neighborhoods.} \label{MCNsbs}  Let $f\colon X\to Y$ be a continuous function between topological spaces. Recall that the \emph{mapping cylinder} $M_f$ of $f$ is the quotient space of the disjoint union $(X\times [0,1])\sqcup Y$ obtained after identifying each $(x,0)\in X\times [0,1]$ with $f(x)\in Y$. By identifying each $x\in X$ with $(x,1)\in M_f$ we may consider $X$ and $Y$ as closed subsets of $M_f$. We will denote the interior of an arbitrary set $A$ in a given topological space by $\inte(A)$. 
Let us now recall the following definition (cf.\ \cite{KwR,Ni,Qu}). 


\begin{definition}
\label{DEF:MCN}	
Let $X$ be a topological space and let $C$ be a closed subset of $X$. A closed neighborhood $N\supset C$ is a \emph{mapping cylinder neighborhood} of $C$ if there exists a surjective function $f\colon \partial N\ra \partial C$ and a homeomorphism $\varphi\colon N\setminus \inte{C}\ra M_{f}$ such that  
	\[ 
	\varphi(p)=
	\left\{\begin{array}{cc} 
	[(1,p)], & \mbox{if}\; p\in \partial N; \\[7pt]
	\;\quad\ \, [p], 	& \mbox{if} \; p\in \partial C.
	\end{array}\right. 
	\]
Let $(M,Z)$ be a relative $C^k$ manifold ($k\geq 2$). A mapping cylinder neighborhood $N\supset Z$ is \emph{$C^l$ regular}, with $l\leq k$, if $\partial N$ is a $C^l$ hypersurface in $M\setminus Z$. 
\end{definition}


\begin{remark} The homeomorphism $\varphi\colon N\setminus \inte(C)\to M_f$ in  Definition~\ref{DEF:MCN} extends to a homeomorphism $\widetilde{\varphi}\colon N\ra M_{\tilde{f}}$, where $\tilde{f}$ is just $f$ with codomain $C$, by letting $\widetilde{\varphi}$ be the identity on $\inte(C)$. We will therefore not distinguish between a mapping cylinder neighborhood of $C$ and a mapping cylinder neighborhood of $\partial C$.
\end{remark}


\section{Gradient flows of {\KL} functions}

We start this section by providing an alternative approach (in Proposition \ref{EX:BOUNDED_LENGTHS}) to proving the uniform boundedness of the lengths of gradient trajectories of a {\KL} function (see \cite[Theorem 2]{K} or \cite[Theorem 18]{BDLM}). 
Recall first the standard theorem on dependence on initial conditions for autonomous (i.e., time independent) ordinary differential equations (cf.\ \cite[Ch.\ 7.3]{HSD}).


\begin{theorem}
\label{T:ODE1} 
Let $X\colon \Omega\ra \bR^n$ be a $C^1$ vector field on an open set $\Omega\subset \bR^n$ and let $x_0\in\Omega$. Let $\varphi_{x_0}$ be the unique solution of the initial value problem
\begin{equation}\label{EQ:IVP_EXISTENCE}\varphi_{x_0}'=X(\varphi_{x_0}),\qquad \varphi_{x_0}(0)=x_0,
\end{equation}
  defined on its maximal  interval $(\omega^-_{x_0},\omega^+_{x_0})$. Then the set
\[ D =\{(t,x_0)\in \bR\times \Omega \mid x_0\in\Omega,\; t\in (\omega^-_{x_0},\omega^+_{x_0}) \}
\]
is open and $\varphi\colon D\ra \Omega$, defined by $\varphi(t,x_0)=\varphi_{x_0}(t)$, is a $C^1$ function.
\end{theorem}

It is natural to ask if anything can be inferred about the continuity of the functions 
\begin{align*} \omega^-\colon\Omega\ra [-\infty,0),&  \quad x_0\mapsto \omega^-_{x_0};\\[5pt]
    \omega^+\colon \Omega\ra (0,\infty],& \quad x_0\mapsto \omega^+_{x_0}.
    \end{align*} 
     
Owing to the classical alternative characterization of upper and lower semicontinuity involving the epigraph of a function, the following result follows immediately from Theorem~\ref{T:ODE1}.


\begin{cor}
\label{COR:OMEGA_SEMICONTINUITY}
The function $\omega^-$ is upper semicontinuous and the function $\omega^+$ is lower semicontinuous.
\end{cor}

One cannot expect better regularity than Corollary~\ref{COR:OMEGA_SEMICONTINUITY}, as  the  example of an open set in $\bR^2$ contained between the graphs of a negative upper semicontinuous function and  a positive lower semicontinuous function with the flow induced by $\partial_y$.

However, the following result provides a link between the left and right end maps $\omega^{\mp}$ for one vector field and its product with a positive function.
 


\begin{lem}
\label{Lem01} 
	Let $\Omega$ be an open subset of $\R^n$, let $X\colon \Omega\ra \bR^n$ be a $C^1$ vector field, and fix $x_0\in\Omega$. Let $h\colon\Omega\ra (0, \infty)$ be a positive, $C^1$ function, set $Y=hX$, and consider the initial value problems
	\begin{equation}
	\label{EQ:IVP_X}
	\varphi_{x_0}'=X(\varphi_{x_0}),\qquad \varphi_{x_0}(0)=x_0
	\end{equation}
	and
	\begin{equation}
	\label{EQ:IVP_Y}
	\psi_{x_0}'=Y(\psi_{x_0}),\qquad \psi_{x_0}(0)=x_0.
	\end{equation}
	Let $D_1\subset \bR\times \Omega$ be the domain of definition of the flow $\varphi$ of $X$ and let $D_2\subset \R\times \Omega$ be the domain of definition of the flow $\psi$ of $Y$.  
	Then there exists a $C^1$ diffeomorphism $\theta\colon D_1\ra D_2$ making the following diagrams commutative:
\[\xymatrix{  D_1\ar[rr]^{\theta} \ar[dr]_{\pi_2} & & D_2\ar[dl]^{\pi_2},\\
  & \Omega &} \qquad\qquad \xymatrix{  D_1\ar[rr]^{\theta} \ar[dr]_{\varphi} & & D_2\ar[dl]^{\psi},\\
  & \Omega &}\]
  where $\pi_2\colon \bR\times \Omega\ra \Omega$ is projection onto the second factor. More precisely, for each $(t,x_0)\in D_1$,
  \[
  \theta(t,x_0)=(\theta_{x_0}(t),x_0)
  \]
  with
  \begin{eqnarray}\label{Eqtheta} \theta_{x_0}(t)=\int_0^t\frac{1}{h(\varphi_{x_0}(r))}~dr.
  \end{eqnarray}
\end{lem}



\begin{proof} Let $\varphi_{x_0}\colon (\omega_{x_0}^{-},\omega_{x_0}^{+})\subset \R\to \Omega$ be the unique solution to the initial value problem \eqref{EQ:IVP_X}.  Then the unique solution $\psi_{x_0}\colon (\nu^{-}_{x_0},\nu_{x_0}^{+})\subset\R\to \Omega$  to the initial value problem \eqref{EQ:IVP_Y}  is given by  $\psi_{x_0}=\varphi_{x_0}\circ \gamma_{x_0}$, where $\gamma_{x_0}\colon(\nu^{-}_{x_0},\nu_{x_0}^{+})\to (\omega_{x_0}^{-},\omega_{x_0}^{+})$ is the unique solution to the initial value problem on the real line given by
\[ 
\gamma_{x_0}'(t)=(h\circ \varphi_{x_0})(\gamma_{x_0}(s)),\qquad \gamma_{x_0}(0)=0.
\]
Hence $\gamma_{x_0}=\theta^{-1}_{x_0}$, where \[
\theta_{x_0}(t)=\int_0^t\frac{1}{h(\varphi_{x_0}(r))}~dr.
\]
Note that $\theta_{x_{0}}\colon (\omega_{x_0}^{-},\omega_{x_0}^{+}) \to\R$ is a $C^1$ diffeomorphism onto its image due to the positivity of $h$. Thus,  we get thus a bijective correspondence between the two initial value problems \eqref{EQ:IVP_X} and \eqref{EQ:IVP_Y}. 
Since the intervals  $(\omega_{x_0}^{-},\omega_{x_0}^{+})$ and $(\nu_{x_0}^{-},\nu_{x_0}^{+})$ on which the solutions $\varphi_{x_0}$ and $\psi_{x_0}$ are defined are maximal, 
it follows that $\theta_{x_0}$ is in fact a $C^1$ diffeomorphism between
$(\omega_{x_0}^{-},\omega_{x_0}^{+})=D_1\cap (\bR\times \{x_0\})$ and $(\nu_{x_0}^{-},\nu_{x_0}^{+})=D_2\cap (\bR\times \{x_0\}$).
The explicit expression for $\theta$ makes it clear that it is $C^1$ also in the $x$ directions. It follows then that $\theta\colon D_1\to D_2$ is a diffeomorphism with the desired properties. 
\end{proof}

Let us illustrate how Lemma~\ref{Lem01} can be of good use.

\begin{prop}
	\label{EX:BOUNDED_LENGTHS}
Let $(M,Z)$ be a relative $C^k$ manifold ($k\geq 2$) and let $f\colon (M,Z)\ra ([0,\infty),\{0\})$ be a relative $C^2$ function such that $\nabla f\neq 0$ on $M\setminus Z$. Assume  that
\begin{eqnarray}\label{EqPsi}\qquad |\nabla_x(\Psi\circ f)|\geq 1,\qquad \forall\;x\in f^{-1}(D_{\Psi})\cap M\setminus Z\end{eqnarray} for some nonnegative, nondecreasing, absolutely continuous  function $\Psi$, where $D_{\Psi}$ is the set of differentiability points for $\Psi$. 
Then the length of the integral curve $\alpha_x$ of $-\nabla f$ with initial condition $x\in M\setminus Z$ satisfies
\begin{align}
\label{Eqnu0} \length(\alpha_x)\leq \Psi(f(x)) \text{ for all } x\in M\setminus Z. 
\end{align}
\end{prop}


\begin{proof}

Let $\varphi_{x_0}$ be an integral curve for $X:=-\frac{\nabla f}{|\nabla f|^2}$ starting at $x_0\in M\setminus Z$. Clearly, $\varphi_{x_0}$ is a reparametrization of $\alpha_{x_0}$. It is well-known (see \cite{Mi1}) that the trajectories for $X$ travel in $t$ units of time from $f$-level $a$ to $f$-level $a-t$. Hence $f(\varphi_{x_0}(r))=f(x_0)-r$. Setting $h=|\nabla f|$, we get from \eqref{EqPsi} that
\begin{align}
\label{Eqh} 
	|h(\varphi_{x_0}(r))|^{-1} & \leq |\Psi'(f(\varphi_{x_0}(r)))|\\[5pt]
							   & =\Psi'(f(x_0)-r)\nonumber
\end{align}
for every $r\in \R$ such that $f(x_0)-r\in D_\Psi$.
Then, equation  \eqref{Eqtheta} in Lemma~\ref{Lem01} and the Fundamental Theorem of Calculus yield that
\begin{align}
\label{Eqtheta2}
	\theta_{x_0}(t)& \leq \int_0^t\Psi'(f(x_0)-r)dr= \int_{f(x_0)-t}^{f(x_0)}\Psi'(u)du 
				 \nonumber \\[5pt]
				   & =\Psi(f(x_0))-\Psi (f(x_0)-t) 
\end{align}
for all $t\leq \omega_{x_0}^+$. Clearly, $\omega_{x_0}^+\leq f(x_0)$ since $f(\varphi_{x_0}(r))\geq 0$ for all $r\leq \omega_{x_0}^+$. Since $\theta_{x_0}$ is an increasing function, (\ref{Eqtheta2}) implies that
\begin{align}
\label{Eqnu}
\nu_{x_0}^+ & =\theta_{x_0}(\omega_{x_0}^+)\leq \theta_{x_0}(f(x_0))  \leq \Psi(f(x_0)). 
\end{align}

Note that $|Y|=1$ for $Y:= h X$, hence $\nu^+_{x_0}$ is the length of the curve $\alpha_{x_0}$, between the initial position until it leaves $M\setminus Z$.

\end{proof}


\begin{rem} 
  Kurdyka's proof in \cite{K} of the uniform boundedness of the lengths of the trajectories is based on the Mean Value Theorem, which, in turn, requires $\Psi$ to be (at least) $C^1$, whereas the proof of Proposition~\ref{EX:BOUNDED_LENGTHS} requires only absolute continuity from $\Psi$.  However, as we will see in Section~\ref{Sec5},
   there is no real  gain in weakening the regularity assumptions on $\Psi$ in Definition~\ref{DEF:KL_FUNCTION}.
\end{rem}

 The next result, inspired by Kurdyka's Proposition 3 in \cite{K}, is the first important ingredient for the proof of Theorem~\ref{T:MAIN_THM}. 


\begin{prop}
\label{propKu}
Let $(M,Z)$ be a relative $C^{k}$ manifold $(k\geq 2)$, let $f\colon (M,Z) \to ([0,\infty),\{0\})$ be a $C^2$, global  {\KL} function with  associated triple $(\rho,U,\Psi)$. For every $x_0\in M\setminus Z$ let $\alpha_{x_0}\colon (\mu_{x_0}^-,\mu_{x_0}^+)\ra M\setminus Z$ be the maximal solution of the Cauchy problem  
		\begin{align*}
		x'   & =-\nabla_xf,\\
		x(0) & =x_0.
		\end{align*}
 Then the following assertions hold:\vspace{5pt}

\begin{enumerate}
	
	\item There exists  an open neighborhood $V\subset f^{-1}([0,\rho))$ of $Z$ such that the integral curves of $-\nabla f$ originating at points in $V\setminus Z$ have forward-limit points on $\partial Z$.\vspace{5pt}
	
	\item The forward-limit mapping
	\begin{align*}
	\Lambda\colon V\setminus Z & \ra \partial Z\\[5pt] 
	x & \mapsto \displaystyle\lim_{t\ra \mu^+_{x}}\alpha_{x}(t)
	\end{align*} 
	 induces  a retraction $R\colon V\ra Z$.\vspace{5pt}
		
		\item The length function 
		\begin{align*}
		L\colon V\setminus Z & \ra (0,\infty)\\
		x & \mapsto \length(\alpha_x\bigr|_{[0,\mu_x^+)})
		\end{align*} 
		is continuous.
   
\end{enumerate} 
\end{prop}


\begin{proof} 


For (1), consider the triple    $(\rho,U,\Psi)$  associated to the global {\KL} function $f\colon M\to [0,\infty)$ and set $g=\Psi\circ f$.  Fix $x_0\in U\setminus Z$ and let $L_{x_0}$ be the length of the curve $\alpha_{x_0}\bigr|_{[0,\mu_{x_0}^+)}$. We know from \eqref{Eqnu0} that $L_{x_0}\leq g(x_0)$. 

Let $\psi_{x_0}\colon [0,L_{x_0})\ra M$ be the arc length reparametrization of $\alpha_{x_0}\bigr|_{[0,\mu_{x_0}^+)}$. 
 
 Let $\widehat{M}$  be the metric completion of $M$. Note that $\psi_{x_0}\colon [0,L_{x_0})\rightarrow M\subset \widehat{M}$ is Lipschitz, hence uniformly continuous.  Since $L_{x_0}<\infty$, by a standard result, the curve $\psi_{x_0}$ has a unique continuous extension
\[ 
	\widehat{\psi}_{x_0}\colon [0,L_{x_0}]\ra \widehat{M}.
\]
 
Since $M$ is locally compact, $M$ is open in $\widehat{M}$, and hence $\partial M=\widehat{M}\setminus M$.  Let  $d(\cdot,\partial M)$ be the distance to $\partial M$ and set $d(x,\partial M)=\infty$ if $\partial M=\varnothing$. Consider now the set
 \begin{align}
 \label{EQ:DEF_V}
 V=\left\{ x\in f^{-1}([0,\rho)) \mid d(x,\partial M)> {g(x)}\right\}.
 \end{align}

 Notice now that $V\subset M$ is open and contains $f^{-1}(0)=g^{-1}(0)$. Indeed, since $M$ is locally compact, $\partial M$ is closed in $\widehat{M}$. Therefore, if $x\in \widehat{M}$ then $d(x,\partial M)=0$ if and only if $x\in\partial M$. 
 
  Let $x_0\in V\setminus f^{-1}(0)$. We show that $\widehat{\psi}(L_{x_0})\in M$. Suppose that this is not the case, i.e.,  $\widehat{\psi}_{x_0}(L_{x_0})\in \partial M$. One gets a contradiction from the following string of inequalities
 \begin{align*} 
 \length(\psi_{x_0})  & =\length(\widehat{\psi}_{x_0})\\[2pt]
 						& \geq d(x_0,\widehat{\psi}_{x_0}(L_{x_0}))\\[2pt]
 						& \geq d(x_0,\partial M)\\[2pt]
 						& > g(x_0) \\[2pt]
 						& \geq \length(\psi_{x_0}).
\end{align*}

Hence $\widehat{\psi}_{x_0}(L_{x_0})\in M$ for every $x_0\in V\setminus Z$. We deduce from the maximality of $\alpha_{x_0}$ that $\widehat{\psi}_{x_0}(L_{x_0})\in Z= f^{-1}(0)=g^{-1}(0)$.  This concludes the proof of assertion (1).


 With $V\subset M$ as in \eqref{EQ:DEF_V} above, consider the mapping $R\colon V\to Z$ given by 
\[
 R(x):=
 \begin{cases}
 \displaystyle\lim_{t\ra \mu_x^+}\alpha_x(t),  & x\in V\setminus Z;\\
												x, & x\in Z.
\end{cases}
\] 

By the proof of item (1),  $R$  is well-defined.  To prove the continuity of $R$ we proceed as follows.  Fix $x_0\in V$, let $\varepsilon>0$, and consider the open ball  in $Z$ centered at $R(x_0)$ given by
\begin{align}
\label{EQ:OMEGA_EPSILON}
\Omega_{\varepsilon}=\{x\in Z\mid d(x,R(x_0))<\varepsilon\}.
\end{align}
If $x_0\in \Int(Z)$, there is nothing to prove.  If $x_0\in \partial Z$, let 
\[
	U=\{x\in V\mid {g(x)}+d(x,x_0)<\varepsilon\}.
\]
Since $g$ is continuous, $U$ is an open subset of $M$ and contains $x_0$. Then, for all $ x\in U$,
\begin{align*}
d(R(x_0),R(x)) & = d(x_0,R(x)) \\[2pt]
& \leq d(x_0,x)+d(x,R(x))\\[2pt]
& \leq d(x_0,x)+\length(\alpha_x)\\[2pt]
& \leq d(x_0,x)+g(x)\\[2pt]
& <\varepsilon. 
\end{align*}
Hence, $R(x)\in \Omega_{\epsilon}$, proving that $R$ is continuous at $x_0\in \partial Z$.

Finally, if $x_0\in V\setminus Z$  let 
$\Omega_{\varepsilon}^c:=Z\setminus \Omega_{\varepsilon}$. Notice that the function $h\colon V\to \R$  given by
\[
	h(y)=d(y, \Omega_{\varepsilon}^c)-{g(y)}
\]	
is continuous  and $h(R(x_0))>0$.
Hence, one can find  $x_1\in V\setminus Z$  close to $R(x_0)$ and on the integral curve  determined by $x_0$, along with a sufficiently small ball $B_1\subset V\setminus Z$ around $x_1$ such that
\begin{align}
\label{eqn1} B_1\subset \{y\in V \mid d(y,\Omega_{\varepsilon}^c)>{g(y)}\}.
	\end{align}

On the other hand, we have, for every $y\in B_1$,
\begin{align*}
  d(y,R(y)) & \leq \length(\alpha_{y})\\[2pt]
  			&  \leq {g(y)}\\[5pt] 
  			& < d(y,\Omega^c_{\varepsilon}).
\end{align*}

This can only mean that $R(y)\in \Omega_{\varepsilon}$ for all $y\in B_1$. 
Now, given $B_1$, there exists an open ball $B_0$ around $x_0$ such that all trajectories that start in $B_0$ will cross $B_1$. Hence the function $R$ will take $B_0$ to $\Omega_{\epsilon}$. Thus, $R$ is continuous also on $V\setminus Z$. Clearly, $R$ is a retraction. This concludes the proof of item (2).

For item (3), we use the notation of Proposition \ref{EX:BOUNDED_LENGTHS}. Observe that $\omega_{x}^+=f(x)$ for every $x\in V\setminus Z$. Note that the length function coincides on $V\setminus Z$ with the function
\begin{align*} 
	\nu^+(x) & =\theta_{x}(\omega_x^+)
		      =\int_0^{f(x)}\frac{1}{h(\varphi_x(r))}~dr \label{EQ:FCN_NU_PLUS},
\end{align*}
where $h=|\nabla f|$.  By  inequality (\ref{Eqh}), we can use the Lebesgue Dominated Convergence Theorem for the continuous family of functions $H_x(r):=h(\varphi_x(r))^{-1}$ to conclude that $\nu^+$ is indeed continuous.
\end{proof}


\begin{remark} If $M$ is a $C^2$ manifold  to begin with and $f\colon (M,Z)\to (\R,\{0\})$ is  a nonnegative $C^{1,1}$ function on $M$, then $\nabla f$ exists, is continuous on $M$, and equals $0$ on the minimal locus $f^{-1}(0)$.
 Therefore, if $f$ is a {\KL} function, the negative gradient flow is forward-complete in $V$. This is because a trajectory of finite length  will stay within a compact subset of the domain of definition of the vector field $-\nabla f$. This vector field is defined everywhere on $V$ and continuous. Hence, the trajectory will exist  for all $t>0$.
\end{remark}

Examining the proof of Proposition \ref{propKu}, one notices that one can easily remove the global \KL{} condition from its hypotheses by taking $V=\bigcup_{p\in \partial Z} V_p$, where 
$V_p$ is the neighborhood on which the retraction $R$ is defined under a local \KL{} nondegeneracy condition.


\begin{cor}
The statements of  Proposition~\ref{propKu} hold for  general {\KL} functions.
\end{cor}


\section{Existence of mapping cylinder neighborhoods}
\label{S:EXISTENCE_MCN}

Proposition~\ref{propKu}  is the main step in the natural generalization of Kurdyka's strong deformation retract result \cite[Proposition 3]{K} to the class of {\KL} functions. In fact, one could have easily gone an extra step already in Proposition~\ref{propKu} in order to recover the full statement from \cite{K}. However, since this is a direct corollary to Theorem~\ref{T:MAIN_THM}, we chose to skip it. One comment is in order. One  consequence of the fact $Z$ has a mapping cylinder neighborhood  is that the pair $(M,Z)$ is an NDR (neighborhood deformation retract) pair (see \cite[Definition 6.2]{KD}), or equivalently
a closed cofibration (provided the space under consideration is compactly generated, for example, Hausdorff and locally compact).  One then could ask the question of whether wildly embedded submanifolds, such as the Alexander horned sphere, could be excluded as the zero locus of a K{\L} function based on this property alone. The answer is no and the reason is that a closed subspace of an ANR (absolute neighborhood retract) is also an ANR if and only if it is a closed cofibration (see \cite[Proposition A.6.7]{FP}).  The Alexander horned sphere, like any  topological manifold, is an ANR (see \cite[Ch.~III, Corollary 8.3]{Hu_ST}).

After these preliminary comments, we are ready to prove Theorem~\ref{T:MAIN_THM}.  Let $f\colon (M,Z)\ra ([0,\infty),\{0\})$ be a  {\KL} (relative) function of class $C^2$. For the following result we will keep the notation of Proposition \ref{propKu}.


\begin{prop} 
\label {P:existH}
There exists a  $C^1$ hypersurface $H\subset V\setminus f^{-1}(0)$ with the following properties: 
\begin{enumerate}
	\item  The hypersurface $H$ intersects transversely every trajectory $\alpha$ of $-\nabla f$ in $V$ exactly once.
	\item The restriction $R\bigr|_{H}\colon H\ra \partial f^{-1}(0)$  of the retraction $R\colon V\to f^{-1}(0)$ is a surjective, proper function.
	\end{enumerate}
\end{prop}


\begin{proof} For item (1),  we will write $f$ for the restriction $f\bigr|_{V}$,  which takes values in $[0,\rho)$. By the Implicit Function Theorem, the hypersurfaces $f^{-1}(c)$ are $C^1$ for $c\in(0,\rho)$.

Define first an ``abstract'' $C^1$ manifold $\mathcal{H}$ of dimension $n-1$ as the inductive limit of  the manifolds $f^{-1}(c)$ as $c\searrow 0$, $c\in (0,\rho)$.   Alternatively, we may describe $\mathcal{H}$ as follows.

 If $c_1>c_2>0$, then the flow function gives an open $C^1$ embedding
\begin{align*}
f^{-1}(c_1) \hookrightarrow f^{-1}(c_2)\\[2pt]
  \qquad x  \mapsto \alpha_x\cap f^{-1}(c_2),
\end{align*}
where $\alpha_x$ is the trajectory determined by $x$.

On $V\setminus f^{-1}(0)$ consider the equivalence relation in which $x\sim y$ if  the two points $x,y$ lie in the same trajectory. 
Take $\mathcal{H}:=(V\setminus f^{-1}(0))/\sim$ with the induced quotient topology. A set $U\subset \mathcal{H}$ is open if and only if $\widehat{U}:=\pi^{-1}(U)$ is a flow-invariant, open subset of $V\setminus f^{-1}(0)$, where 
\[
\pi\colon V\setminus f^{-1}(0)\ra \mathcal{H}
\] is the canonical quotient map. Note that the topology is  Hausdorff. Moreover, since $f$ has no critical points  in $V\setminus f^{-1}(0)$, the flow is locally a product, which implies that $\mathcal{H}$ is locally homeomorphic to $\bR^{n-1}$. Now, we need an atlas. For a fixed $c \in (0,\rho)$, we take the usual charts  in $f^{-1}(c)$  and project them via $\pi$ to $\mathcal{H}$.  It is not hard to verify that these charts are $C^1$-compatible, because the flow function is $C^1$, thus concluding that $\mathcal{H}$ is a $C^1$ manifold. Moreover, $\pi\colon V\setminus f^{-1}(0)\ra \mathcal{H}$ is a $C^1$ function. 

Let $h\colon\mathcal{H}\ra [0,\infty)$ be an exhaustion function, i.e., a continuous, proper  function. Let $\mathbb{N}/2:= \{n/2 \mid n\in \mathbb{Z}_{\geq 0}\}$ and,  for every $k\in \mathbb{N}/2$, take the covering of $\mathcal{H}$ by the relatively compact, open sets
\[ 
A_{2k+1}:=h^{-1}((k-1/3,k+1/3)).
\]
We let $\{\phi_{2k+1}\}_{k\in \mathbb{N}/2}$ be a $C^1$ 
partition of unity subordinate to the covering $\{A_{2k+1}\}_{k\in  \mathbb{N}/2}$ of $\mathcal{H}$. Note that  each point $x\in\mathcal{H}$ belongs to at most two of the sets $A_{2k+1}$ and, consequently, at most two of the functions $\phi_{2k+1}$ satisfy $\phi_{2k+1}(x)\neq 0$. In fact, one can find a sufficiently small neighborhood of $x$ that intersects at most two of the sets $A_{2k+1}$ nontrivially.
By composing the $\phi_{2k+1}$ with $\pi$, we get $C^1$ functions
\begin{align*}
\widehat{\phi}_{2k+1}\colon V\setminus f^{-1}(0)\ra [0,1],\\[2pt]
\widehat{\phi}_{2k+1}:=\phi_{2k+1}\circ \pi.
\end{align*}
We will also consider the open sets $\widehat{A}_{2k+1}:=\pi^{-1}(A_{2k+1})$ in $V\setminus f^{-1}(0)$. For $n\in \mathbb{N}$, let
\[ 
	\widehat{U}_{n}:=\bigcup_{2k+1\leq n} \widehat{A}_{2k+1}.
\]
Clearly, the sets $\widehat{U}_n$ are not  relatively compact in $f^{-1}(0,\rho)$. 

We will need the following properties of compact subsets $K$ of $\mathcal{H}$, which we will presently prove:

\begin{eqnenumerate}
	\item There exists $c>0$ such that $K\subset \pi(f^{-1}(c))$. In particular, every trajectory of $-\nabla f$ in $\pi^{-1}(K)$ intersects $f^{-1}(c)$.\label{IT:COMPACT_1}
	\item The intersection $\pi^{-1}(K)\cap f^{-1}(c)$ is a nonempty compact set in $f^{-1}(c)$. \label{IT:COMPACT_2}
\end{eqnenumerate}

 To prove \ref{IT:COMPACT_1}, observe that the images of all preimages $f^{-1}(c)$ via $\pi$ are open subsets of $\mathcal{H}$. Their union covers $\mathcal{H}$ and, since $K$ is compact, there will exist a $c$ such that $K\subset \pi(f^{-1}(c))$. Indeed, first, compactness implies that $K$ is contained in the union of finitely many sets $\pi(f^{-1}(a_i))$, with $a_i\in (0,\rho)$, $1\leq i\leq k$ for some positive integer $k$. We can arrange for the $a_i$ to satisfy $a_{i+1}<a_{i}$, which implies that $ \pi(f^{-1}(a_{i}))\subset\pi(f^{-1}(a_{i+1}))$. It follows that $\bigcup_{i=1}^k\pi(f^{-1}(a_i))= \pi(f^{-1}(a_k))$. The second claim in \ref{IT:COMPACT_1} is now an immediate consequence.
	
	To verify \ref{IT:COMPACT_2}, note that $\pi\bigr|_{f^{-1}(c)}$ is a homeomorphism onto its image  and 
	\[
	\pi^{-1}(K)\cap f^{-1}(c)=(\pi\bigr|_{f^{-1}(c)})^{-1} (K).
	\]

By properties \ref{IT:COMPACT_1} and \ref{IT:COMPACT_2}, since the sets $U_n:=\bigcup_{2k+1\leq n} A_{2k+1}$ are relatively compact in $\mathcal{H}$, there exists  a decreasing sequence $\rho>c_1>c_2>\ldots$ of positive numbers with  $c_n\searrow 0$ and satisfying the following conditions:

 \begin{itemize}
 \item The sets $\widehat{U}_n\cap f^{-1}(c_n)$ are relatively compact in $f^{-1}(c_n)$.\smallskip
 
 \item Every trajectory of $-\nabla f$ in $\widehat{U}_n$ intersects $f^{-1}(c_n)$.
 \end{itemize}

Define a positive function $\widehat\Phi$ on $V\setminus f^{-1}(0,\rho)$ by letting
\[ \widehat{\Phi}:=\sum_{k\in \bN/2} \frac{\widehat{\phi}_{2k+1}}{c_{2k+1}}.
\]
Recall that, by construction,  at any given point in $V\setminus f^{-1}(0)$, at most two of the functions $\widehat{\phi}_{2k+1}$ are nonzero.
Thus, the function  
\begin{equation}
\label{ftilde}
\widehat{f}:= f\widehat{\Phi}
\end{equation}
is well-defined and $C^1$ on $V\setminus f^{-1}(0)$.
Note that $\widehat{\Phi}$ is constant on the trajectories of $-\nabla f$, which implies that 
\[
	\langle\nabla \widehat{\Phi},\nabla f\rangle=d\widehat{\Phi}(\nabla f)=0.
\]
Hence 
\[\langle\nabla \widehat{f},\nabla f\rangle=\widehat{\Phi}|\nabla f|^2>0 \mbox { on }  V\setminus f^{-1}(0).
\]
Therefore, $\widehat{f}$ is strictly decreasing along the trajectories of $-\nabla f$. Moreover, $\widehat{f}$ can be continuously extended by $0$ along $Z$.

Define $H:=\widehat{f}^{-1}(1)$. By what was just said, \emph{if} a trajectory of $-\nabla f$ intersects $H$ it will do so only once and transversely.  

Let $\alpha\colon(\mu^-,\mu^+)\ra V\setminus f^{-1}(0)$ be a trajectory of $-\nabla f$. Clearly,
\[
\lim_{t\ra \mu^+}\widehat{f}(\alpha(t))=0.
\]
We will now show that
\begin{equation}
\label{lm1}
\lim_{t\searrow \mu^-}\widehat{f}(\alpha(t))>1.
\end{equation}
From (\ref{lm1}) and the fact that $\widehat{f}$ is strictly decreasing, it follows that there exists $t_0\in\R$ such that $\widehat{f}(\alpha(t_0))=1$ or, equivalently, that every trajectory $\alpha$ intersects $H$.

 In order to prove (\ref{lm1}), let $i\in \bN/2$ such that $p=\pi(\alpha(t))\in  A_{2i+1}\cap A_{2i+2}$. We have
 \begin{align}
 \label{EQ:F_HAT}
 \widehat{f}(\alpha(t))=f(\alpha(t))\widehat{\Phi}(\alpha(t))=f(\alpha(t))\left(\frac{\phi_{2i+1}(\pi(\alpha(t)))}{c_{2i+1}}+\frac{\phi_{2i+2}(\pi(\alpha(t)))}{c_{2i+2}}\right).
\end{align}
 By the choice of the constants $c_i$ we know that the trajectory $\alpha$ which lies in $\widehat{U}_{2i+1}$ intersects $f^{-1}(c_{2i+1})$ at a time $t'$. Hence, for $t<t'$ we have $f(\alpha(t))>f(\alpha(t'))=c_{2i+1}$. Therefore, for $t<t'$ we have,  by \eqref{EQ:F_HAT},
 \begin{align*}
 \widehat{f}(\alpha(t))  >c_{2i+1}\left(\frac{b(t)}{c_{2i+1}}+\frac{1-b(t)}{c_{2i+2}}\right)
 					 > c_{2i+1}\frac{b(t)+(1-b(t))}{c_{2i+1}}
 					 =1,
  \end{align*}
 where $b(t)\equiv \phi_{2i+1}(\pi(\alpha(t)))$.
  This proves \eqref{lm1} in the case where $p\in  A_{2i+1}\cap A_{2i+2}$. The case where $p\in A_{2i+1}\setminus (A_{2i}\cup A_{2i+2})$ is dealt with in a similar fashion. This finishes the proof of item (1).

We now prove item (2). Let us first show that $R\bigr|_{H}\colon H\to \partial f^{-1}(0)$ is proper. 
Granted this, since a proper continuous function to a locally compact Hausdorff space is closed, it will suffice to prove that the image of $R\bigr|_{H}$ is dense to conclude that $R\bigr|_{H}$ is also surjective.

 We notice first that the continuous bijection $\pi\bigr|_{H}:H\ra \mathcal{H}$ is a homeomorphism. This can be seen by noting that the flow function produces an open embedding $f^{-1}(c)\hookrightarrow H$ for every $c\neq 0$  and, for every point $p\in\mathcal{H}$, one finds an open neighborhood of $p\in \pi(f^{-1}(c))$. 
 
 Let $K\subset f^{-1}(0)$ be compact. By the previous paragraph, it will be enough to prove that, for small $a>0$,   the preimage $R^{-1}(K)\cap f^{-1}(a)=(R\bigr|_{f^{-1}(a)})^{-1}(K)$
 is compact in $f^{-1}(a)\subset V$.  Since the level sets of $f$ equal the level sets of $g$, we can reformulate the last sentence with $g$ instead of $f$. Since $R^{-1}(K)\cap g^{-1}(a)$ is closed, the only way it can fail to be compact is if there exists a sequence $\{x_n\}\subset R^{-1}(K)\cap g^{-1}(a)$ such that $
 x_n\ra \partial g^{-1}(a)$. By this we mean that $x_n$ eventually gets out  of every compact subset of $g^{-1}(a)$. Then, clearly,
  \begin{align}
  \label{eqf1}
  \liminf d(x_n,K)\geq  d(K, \partial g^{-1}(a))\geq d(K,\partial V)>0,
  \end{align}
 where $\partial V:=\widehat{V}\setminus V$ is the boundary of $V$ within its metric completion. The last inequality is a consequence of the local compactness of $M$ and, therefore, of its open subset $V$. Indeed, one has $d(x,\partial V)>0$ if $x\in V$, and it follows easily that $d(K,\partial V)>0$ for every compact  $K\subset V$. 
  
  Take now $a_1>0$ such that
 \[ {a_1}< {d(K,\partial V)}.
 \]
By \eqref{Eqnu},
 the distance  to $K$ from any point $p$ in $R^{-1}(K)\cap g^{-1}(a_1)$  (which, in particular, determines a trajectory with end-point in $K$) will be at most $a_1$ and hence strictly smaller than $d(K,\partial V)$. Therefore, \eqref{eqf1} cannot be fulfilled for any sequence of points in $R^{-1}(K)\cap g^{-1}(a_1)$. Thus, we conclude that $R\bigr|_{H}$ is proper.

Let us now prove that the image of $R\bigr|_{H}$ is dense in $f^{-1}(0)$. By item (1), the image of $R\bigr|_{H}$ equals the image of $R$. Take a point $p\in f^{-1}(0)$ and take an open  ball around $p$. Then, by the uniform bound on the lengths of trajectories of $-\nabla f$, there exists a smaller ball around $p$ such that  all the trajectories starting in the smaller ball cannot leave the bigger ball. Hence the image of the limit function $R$ is dense and this concludes the proof.

\end{proof}


 The surjectivity of $R$ has the following trivial consequence.
 \begin{cor}
 	\label{csL} 
 	If $p\in \partial Z$ is K{\L} nondegenerate, then there exists a $C^1$ relative curve $\gamma\colon ([0,\epsilon),\{0\})\ra (M,Z)$ such that $|\gamma'|\leq 1$ and $\gamma(0)=p$. 
 \end{cor}
 \begin{proof} Take $\gamma$ to be  the parametrization by arc length of the integral curve of $-\nabla f$ having  $p$ as a limit point.
 \end{proof}

We come now to the main result in this section, which finishes the proof of Theorem~\ref{T:MAIN_THM}. Recall that the function $\widehat{f}$ was defined in \eqref{ftilde} above.


\begin{prop} \label{MT}The set $\widehat{f}^{-1}([0,1))$  is a $C^1$ regular
	 open mapping cylinder neighborhood of $f^{-1}(0)$.
\end{prop}
\begin{proof} Recall that $Z:=f^{-1}(0)=\widehat{f}^{-1}(0)$. Let $M_R:=H\times [0,1]\sqcup \partial Z/_{\sim}$ be the closed mapping cylinder of $R\bigr|_{H}:H\ra  \partial Z$. Denote $\partial^{[0,1]}\widehat{f}:=\widehat{f}^{-1}([0,1])\setminus \Int f^{-1}(0)=\widehat{f}^{-1}(0,1]\cup \partial Z$.

The right-end function $q\ra \omega_q^+$ for the initial value problem for the vector field $-\frac{\nabla f}{|\nabla f|^2}$ is  a continuous function from $V\setminus f^{-1}(0)$ to $(0,\infty)$. In fact, $\omega_q^+=f(q)$.

We can then fix a homeomorphism
\begin{align}
\label{EQ:HOMEO__NU}
\upsilon\colon [0,1]\times H\ra \{(t,q)\mid q\in H, \; t\in[0,\omega_q^+]\}\subset \bR\times H
\end{align}
such that $\pi_2(\upsilon(t,q))=q$ and $\upsilon(0,q)=(0,q)$.

Define the function
\begin{align*}
	\Phi\colon M_R & \ra \partial^{[0,1]}\hat{f}\\[5pt]
	\Phi(q,t) & =
	\begin{cases}
	\gamma_q(\upsilon(1-t,q)),& \text{ if } t\in (0,1],\\
	\displaystyle\lim_{s\ra \omega_q^+} \gamma_q(s), & \text{ if } t=0,
	\end{cases}
\end{align*}
where $\gamma_q$ is the solution of the Cauchy problem for $-\nabla f/|\nabla f|^2$ with initial value $q\in H$. Note that $\Phi$ is defined at points $p=R(q)\in \partial Z$, in which case 
$\Phi(p)=\Phi(q,0)=R(q)=p$.
We also have that $\Phi\bigr|_{H\times \{1\}}=\id_{H}$.

 Clearly, $\Phi\bigr|_{H\times (0,1]}$ is a continuous bijection onto $\widehat{f}^{-1}(0,1]$, while $\Phi$ is the identity on $\partial Z$. Hence, $\Phi$ is a bijection.
 It remains to prove that $\Phi$ and $\Phi^{-1}$ are continuous.
 
 Recall that, by item (2) in Proposition~\ref{P:existH}, the map $R\bigr|_{H}$ is surjective. Then, we may identify the mapping cylinder $M_R$ with $H\times [0,1]/_{\sim}$, where  $(q_1,0)
 \sim(q_2,0)$ if and only if $R(q_1)=R(q_2)$.  Modulo this identification, and modulo the homeomorphism $\upsilon$ defined in \eqref{EQ:HOMEO__NU}, the function $\Phi$ is the induced bijection  of the extension to $\{(\omega_q^+,q)~|~q\in H\}$ of the flow function of $-\frac{\nabla f}{|\nabla f|^2}$, a priori defined  on 
 \[
 D:=\{(t,q)\mid q\in H, \; t\in[0,\omega_q^+)\}.
 \] 
 More precisely, if $\phi\colon D\ra \partial^{[0,1]}\hat{f}$ is the flow function, then $\phi$ extends \emph{continuously} to a surjective function 
 \[ 
 \widetilde\phi\colon\tilde{D}\ra \partial^{[0,1]}\hat{f},
 \] 
 where $\widetilde{D} :=\{(t,q)\mid q\in H, \; t\in[0,\omega_q^+]\}$. The continuity of $\widetilde{\phi}$ is proved in the same manner as the continuity of the retraction $R\colon V\ra \partial Z$ in Proposition~\ref{propKu}, while surjectivity is clear. By a standard result in point-set topology, this induces a continuous bijection $\tilde{D}/_{\sim'}\ra \partial^{[0,1]}\hat{f}$, where $\sim'$ identifies the points in the same fiber, i.e., \emph{mutatis mutandis}, $\sim'\equiv\sim$ in $H\times [0,1]/_{\sim}$. Adding to this the fact that $\tilde{\phi}$ is proper, which follows along the same lines as the properness of $R\bigr|_{H}$ in the proof of Proposition~\ref{P:existH}, we  conclude the continuity of the inverse, and hence of $\Phi^{-1}$.
  \end{proof}

\section{Remarks on the K{\L} nondegeneracy condition}
\label{Sec5}

The {\KL} condition was introduced by Kurdyka in \cite{K} in the o-minimal context, where it served as the appropriate generalization of the \L ojasiewicz inequality.  The characterization of the {\KL} condition was thoroughly studied in a nonsmooth, infinite dimensional setting in \cite{BDLM}. Theorems 18 and 20 in \cite{BDLM}, characterizing the {\KL} property, are stated for proper, lower semicontinuos, semi-convex functions on a Hilbert space.  The context of relative manifolds that we treat here is slightly different. In particular, in this section we do not assume that $f$ is semi-convex, but rather ask for $f$ to be $C^1$ on the regular part $M\setminus Z$.  Therefore one may think of the results in this section, in which we prove Theorem~\ref{T:MAIN_THM2}, as complementing those in \cite{BDLM}. As a matter of fact, one key lemma about upper semicontinuous functions  from \cite{BDLM} appears prominently here as well. We start by stating this result in a form that applies to  lower semicontinuous functions too. 


\begin{lem}[\protect{cf.~\cite[Lemma 45]{BDLM}}] \label{Lem45} Let $u\colon (0,r_0]\ra (0,\infty)$ be a  locally integrable,
semicontinuous function. Then there exists a continuous function $w\colon (0,r_0]\ra (0,\infty)$ such that
\begin{enumerate}
\item $w\leq u$, if $u$ is lower semicontinuous; 
\item  $u\leq w$, if $u$ is upper semicontinuous; and
\item $|u-w|\in L^1(0,r_0)$.  
\end{enumerate}
\end{lem}


\begin{remark}
Note that the local integrability condition in the statement of the lemma is automatically satisfied when $u$ is upper semicontinuous, in which case $u$ is clearly locally bounded above.
\end{remark}


\begin{proof}[Proof of Lemma \ref{Lem45}] We follow the  proof of  Lemma 45  in \cite{BDLM} and include the details when $u$ is lower semicontinuous. 

 Let $\{a_k\}\subset (0,r_0)$ be a strictly decreasing sequence with $a_k\searrow 0$. We may justify the existence, for each $k$, of a continuous function $w_k\colon [a_{k+1},a_k]\ra (0,\infty)$ satisfying
\begin{align}
 w_k\leq u \label{IT:LEM5.1_A}
\end{align}
and
\begin{align}
	\int_{a_k}^{a_{k+1}} w_k\geq \int_{a_k}^{a_{k+1}}u -\frac{1}{(k+1)^2} \label{IT:LEM5.1_B}	
\end{align}

using \emph{Moreau envelopes } $e_{\lambda}^{u}$
for lower semicontinuous functions (see \cite[Ch.\ 1.G]{RW}). These are given by
\[
	e_{\lambda}^u(x)=\inf_{w}\left\{u(w)+\frac{1}{2\lambda}|w-x|^2\right\}\leq u(x).
\] 
The envelope $e_{\lambda}^u$ of a strictly positive lower semicontinuous function $u$  defined on a compact set is a strictly positive continuous  function.  When $\lambda\searrow 0$, one has $e_{\lambda}^u(x)\nearrow u$ pointwise. The  integrability of $u$  on $[a_{k+1},a_k]$ implies that $e_{\lambda}^u{{L^1\atop \xrightarrow{\hspace*{0.4cm}}}\atop \;} u$, when $\lambda \searrow 0$, so one can choose $\lambda$ such that \eqref{IT:LEM5.1_B} is satisfied for $e_{\lambda}^u=:w_k$.

  Let $ m_k:=\min\left\{\frac{w_{k-1}(a_{k})}{ w_k(a_k)},\frac{w_{k}(a_{k})}{ w_{k-1}(a_k)}\right\}$. Clearly, $0<m_k\leq 1$. 
    We construct a continuous function $w\colon (0,r_0]\ra (0,\infty)$ by induction on $k$ by letting
 \[
  w= 
  \begin{cases}
   w_k & \text{on } [a_{k+1},a_{k}-\varepsilon_k],\\
  \lambda_kw_k & \text{on } [a_k-\varepsilon_k,a_k],
  \end{cases} 
 \]
 where $\lambda_k\colon [a_k-\varepsilon_k,a_k]\ra[m_k,1]$ is the unique affine function satisfying (necessarily)  \linebreak $\lambda_k(a_k-\varepsilon_k)=1$, $\lambda_k(a_k)=m_k$, and where $\varepsilon_k>0$ is to be chosen later.
Then 
 \begin{align*}
 \int_{a_{k+1}}^{a_{k}-\varepsilon_k} w_k(r)~ dr& =\int^{a_k}_{a_{k+1}}w_k(r)~dr-\int_{a_k-\varepsilon_k}^{a_k}w_k(r)~dr \\[5pt]
 & \geq \int^{a_k}_{a_{k+1}}u(r)~dr-\frac{1}{(k+1)^2}-\int_{a_k-\varepsilon_k}^{a_k}w_k(r)~dr.
 \end{align*}
 Hence, for any choice of $\varepsilon_k>0$, we have
 \begin{eqnarray}\label{eq33}\int^{a_k}_{a_{k+1}}w(r)~dr\geq  \int^{a_k}_{a_{k+1}}u(r)~dr-\frac{1}{(k+1)^2}+\int_{a_k-\epsilon_k}^{a_k}(\lambda_k-1)w_k(r)~dr.
 \end{eqnarray}
 Choose now $0<\varepsilon_k<a_{k}-a_{k+1}$ sufficiently small so that
 \begin{align}
\label{EQ:EPSILON_K}
 \int_{a_k-\varepsilon_k}^{a_k}(1-\lambda_k(r))w_k(r)~dr\leq \frac{1}{(k+1)^2}.
 \end{align}
Since, for any choice of $\varepsilon_k>0$, one has $|(1-\lambda_k(r))w_k(r)|\leq \displaystyle\max_{[a_{k+1},a_k]}w_k$ for all $r\in [a_{k}-\varepsilon_k,a_k]$,  we may choose an $\varepsilon_k$ satisfying \eqref{EQ:EPSILON_K}. Hence, by inequality \eqref{eq33}, we get
 \[\int^{a_k}_{a_{k+1}}w(r)~dr\geq  \int^{a_k}_{a_{k+1}}u(r)~dr-\frac{2}{(k+1)^2}.
 \]
Therefore, by induction,
 \[
  \int_{a_{k+1}}^{r_0}w(r)~dr\geq  \int_{a_{k+1}}^{r_0}u(r)~dr-2\sum_{i=1}^{k+1}i^{-2}.
 \]
 Since, on every $[a_{k+1},a_k]$, we have $w\leq w_k\leq u$, we get that 
 \[
 0\leq \int_0^{r_0} (u(r)-w(r))~dr\leq 2\sum_{{i=1}}^{{\infty}}i^{-2}.
 \]
\end{proof}

	With Lemma~\ref{Lem45} and the following definition in hand, we are ready to prove Theorem~\ref{T:MAIN_THM2}, the second main result of this note. 
	
	
	  \begin{definition}
	  \label{DEF:SIM_WKL NONDEG_PT}
	  	
	  Let $(M,Z)$ be a relative manifold and let $f\colon (M,Z)\ra ([0,\infty),\{0\})$ be a relative $C^1$ function.
		\begin{enumerate}
			\item A point $p\in \partial Z$ with a neighborhood $U_p$ such that $\nabla_x f\neq 0$ for all $x\in U_p\setminus Z$ will be called \emph{simple nondegenerate}.
			\item 
			A simple nondegenerate point $p\in \partial Z$ is \textit{weakly K\L-nondegenerate }if there exists a triple $(\rho,U,\Psi)$ with $\rho>0$, $U\subset M$ an open neighborhood of $p$,
			and  $\Psi\colon([0,\rho],\{0\})\ra ([0,\infty),\{0\})$  a relative function  which is nondecreasing and absolutely continuous  on $[0,\rho]$, and   such that  
			\begin{equation}
			\label{wkl}|\nabla_x(\Psi\circ f)|\geq 1 \text{ for all }~x\in f^{-1}(D_{\Psi})\cap U,
			\end{equation}
			where $D_{\Psi}\subset (0,\rho)$ are the points of differentiability of $\Psi$.
		\end{enumerate}  
	\end{definition}

\subsection*{Proof of Theorem~\ref{T:MAIN_THM2}} For ease of reference, let us recall the statement of the theorem.
  

  \begin{theorem}\label{KLchara} The following are equivalent for a simple nondegenerate point $p\in \partial Z$:
  \begin{enumerate}
  \item The point $p\in \partial Z$ is  K{\L} nondegenerate,\label{T:KLchara_A}
  \item The point $p\in \partial Z$ is weakly K{\L} nondegenerate,\label{T:KLchara_B}
  \item There exists a compact neighborhood $K\ni p$ such that the  upper semicontinuous  function
  \[\alpha^K(t)=\displaystyle{\frac{1}{\displaystyle\inf_{x\in f^{-1}(t)\cap K_p}|\nabla_x f|}}
  \]
   is in $L^1(0, \rho)$ for some $\rho>0$. \label{T:KLchara_C}
  \item There exists an open  neighborhood $U\ni p$, a positive number $\rho>0$,  and a continuous function $a\colon (0,\rho]\ra (0,\infty)$ such that  $a^{-1}\in L^1(0,\rho)$, and
   \begin{align}\label{eq27_BIS} |\nabla_xf|\geq a(f(x)) \text{ for all } x\in U\setminus Z.
   \end{align}
   \label{T:KLchara_D}
  \end{enumerate}
  \end{theorem}


     \begin{proof} 
     To show that $\eqref{T:KLchara_A}\Rightarrow \eqref{T:KLchara_B}$, it is enough to prove that a continuous and increasing $\Psi$ which is $C^1$ on $(0,\rho)$  is in fact absolutely continuous on $[0,\rho/2]$, since condition (K{\L}2) in Definition \ref{DEF:KL_FUNCTION} will trivially imply condition (\ref{wkl}). For this, it is enough to check that the relation
  \[ 
  	\Psi(t)-\Psi(s)=\int_{s}^t\Psi'(r)~dr,
  \]
   valid a priori for $\rho/2\geq t\geq s>0$, also holds for $s=0$. Since $\Psi'\geq 0$, it follows that for any decreasing sequence $s_n\searrow 0$ the sequence of functions $g_{n}(r):=\xi_{[s_n,t]}\Psi'(r)$ is pointwise increasing and nonnegative. Hence, by the Monotone Convergence Theorem,
   \[ 
   \int_{s_n}^t\Psi'(r)~dr=\int_{[0,t]} g_n(r)~dr\longrightarrow \int_0^t\Psi'(r)~dr.
   \]
Due to the continuity of $\Psi$ in $0$, we get 
\[\Psi(t)-\Psi(0)=\Psi(t)-\lim_{n\ra \infty}\Psi(s_n)=\lim_{n\ra \infty}\int_{s_n}^t\Psi'(r)~dr=\int_0^t\Psi'(r)~dr.
\]
By the Fundamental Theorem of Calculus, $\Psi$ is absolutely continuous on $[0,\rho/2]$.

To prove that $\eqref{T:KLchara_B}\Rightarrow \eqref{T:KLchara_C}$, note first that the upper semicontinuity of the function
\[
\alpha^K(t)=\sup_{x\in f^{-1}(t)\cap K}\frac{1}{|\nabla_x f|}
\]
for every compact set $K$ is straightforward.
Since $|\Psi'|=\Psi'$ on $D_{\Psi}$, one concludes that $|\nabla_x(\Psi\circ f)|\geq 1$ implies that $\Psi'(f(x))\geq \frac{1}{|\nabla_x f|}$ for all $x\in f^{-1}(D_{\Psi})\cap U$. Therefore,
\[ 
	\Psi'(t)\geq \displaystyle\sup_{x\in f^{-1}(t)\cap U}\frac{1}{|\nabla_x f|} \text{ for all } t\in D_{\Psi}.
\]
   Clearly, this implies that for every compact neighborhood $K\ni p$ with $K\subset U$ we have 
   \[
   \Psi'(t)\geq \alpha^K(t).
   \]
   Integrating on $[0,\rho]$ and using the Fundamental Theorem of Calculus, we get that $\alpha^K\in L^1(0,\rho).$
   
   For $\eqref{T:KLchara_C}\Rightarrow \eqref{T:KLchara_D}$, we use Lemma~\ref{Lem45} with $u=\alpha^K$, for which we need that $\alpha^K$ is finite valued. Then there exists continuous function $\widetilde{a}\colon (0,\rho]\ra (0,\infty)$ such that $\widetilde{a}\geq \alpha^K$ and $\int_0^\rho \tilde{a}<\infty$.
    Let $a:=1/\tilde{a}$.
    This means that
   \[
    |\nabla_x f|\geq \inf_{x\in f^{-1}(f(x))\cap K} |\nabla_x f|=(\alpha^K(f(x)))^{-1}\geq a(f(x)) \text{ for all }  x\in K.
   \]
   To obtain the desired conclusion, we take $U$ to be the interior of $K$.
   
   Finally, we show  $\eqref{T:KLchara_D}\Rightarrow \eqref{T:KLchara_A}$. Define, for $t\geq 0$,
   \[
   \Psi(t)=\int_0^{t} \frac{1}{a(r)}~dr.
   \]
   Clearly, $\Psi$ is  increasing and $\Psi(0)=0$. For $t>0$, one has $\Psi'(t)=1/a(t)$, since we can write 
   \[
	\Psi(t)=\int_s^{t} \frac{1}{a(r)}~dr+\int_0^s \frac{1}{a(r)}~dr
   \] 
   for some fixed $0<s<t$. In other words, $\Psi\big|_{(0,\rho)}$ is $C^1$. Then \eqref{eq27_BIS} can be written as 
   \[
   |\nabla_x f|\geq \frac{1}{\Psi'(f(x))},
   \]
   which is the same as $|\nabla_x(\Psi\circ f)|\geq 1$.
  \end{proof}

\begin{rem} The equivalence of items \eqref{T:KLchara_A} and \eqref{T:KLchara_C} in Theorem~\ref{KLchara} appears also in \cite[Theorems 18 and 20]{BDLM} when $f$ is assumed semi-convex, but only lower semicontinuous.
\end{rem}


\begin{rem} 
	The continuity of the function $a$ in item \eqref{T:KLchara_D} in Theorem~\ref{KLchara} is not essential. In fact, one can substitute it with the following condition.
  	\begin{itemize}
  	\item[$(4')$] There exists an open  neighborhood $U\ni p$, a positive number $\rho>0$, and a Lebesgue measurable function $a\colon(0,\rho]\ra (0,\infty)$ such that  $a^{-1}\in L^1(0,\rho)$ and
   \begin{align}
   \label{eq28}\qquad |\nabla_xf|\geq a(f(x)) \text{ for almost all }   x\in U\setminus Z
   \end{align}
  	\end{itemize}
  	Indeed, we have the implications $\eqref{T:KLchara_D}\Rightarrow (4')\Rightarrow (2')\Rightarrow \eqref{T:KLchara_C}$, where $(2')$ is condition $\eqref{T:KLchara_B}$ in Theorem~\ref{KLchara} with \eqref{wkl} weakened to hold for almost all $x\in\; f^{-1}(D_{\Psi})\cap U$. By examining the proof of Theorem~\ref{KLchara}, we get $(2')\Rightarrow \eqref{T:KLchara_C}$, while $(4')\Rightarrow (2')$ follows the same path as $\eqref{T:KLchara_D}\Rightarrow \eqref{T:KLchara_A}\Rightarrow \eqref{T:KLchara_B}$.
   	One has to be careful that, in \eqref{T:KLchara_D}, the condition that $p$ is simple nondegenerate is redundant. This is not the case for $(4')$ as stated, which is why for Theorem~\ref{KLchara} to hold one should keep simple nondegeneracy as a ``meta-condition''.  Without it,  $(4')$  is what some authors   call \emph{gradient inequality} (see, for example, \cite[Definition 2.2]{Hu} and Remark \ref{Finrem} below).
  \end{rem}


\begin{remark}
	\label{remtrans} 
	Note that condition \eqref{T:KLchara_D} is trivially satisfied by a transnormal function since, in that context (see \cite[Lemma 1]{Wa}), the integral $\int_{c}^{d} a^{-1}$ represents the distance between the parallel level sets $c$ and $d$ of $f$. In this case the integral is finite because the distance is finite. 
\end{remark}
 
 
  \begin{remark}  We may apply Theorem~\ref{KLchara} \eqref{T:KLchara_D} when $U\setminus Z$ has a finite number of connected components $U_1,\ldots, U_k$ and, on each $U_i$, one can prove the existence of a continuous function $a_i>0$ such that $\int_0^{\rho} a_i^{-1}<\infty$
  and \eqref{eq27_BIS} is satisfied on $U_i$.  Take $a:=\min\{a_1,\ldots,a_k\}$ and note that one has $a^{-1}\in L^1(0,\rho)$. Indeed, since $\sum_{i=1}^k a_i^{-1}\in L^1(0,\rho)$,  the claim follows from the fact that
  \begin{align*}
   \frac{1}{a}<\sum_{i=1}^k\frac{1}{a_i}.
  \end{align*}
  Therefore, the function $a$ satisfies \eqref{eq27_BIS} on all of $U$.  
 For example, when $f^{-1}(c)$ is a critical level set of a Morse function $f\colon M\ra \bR$, then $|f-c|$ is a K{\L} function, since this is true for the nonnegative functions $(f-c)^+:=\max\{f-c,0\}$ and $(f-c)^-:=-\min\{f-c,0\}$ (see Example~\ref{examples} (5)).
  \end{remark}
  
  Here is a simple, if somewhat straightforward, application of what we have done so far.
  
  
\begin{cor} Let $v\in \bR^n$ be a unit-norm vector and let $M_0\subset \R^n$ be an oriented $C^2$ hypersurface with unit normal $N_x$ at $x\in M_0$ contained in the open half-space 
\[
  H_{v}^+:=\{x\in \bR^n\mid \langle x,v\rangle> 0\}.
\]
Let $H_{v}:=\partial H_v^+$
and suppose that $p\in \overline{M_0}\cap H_{v}$ is a point such that there exists a triple $(\rho,U,a)$, where $\rho>0$, $U\ni p$ is a neighborhood, and $a\colon(0,\rho)\ra (0,\infty)$ is a continuous  function with $1/a\in L^1(0,\rho)$
 and satisfying 
\begin{align*} |v-\langle v,N_x\rangle|\geq a(\langle v,x\rangle)
\end{align*}
for all  $x\in M_0\cap U\cap\{x\mid \langle v,x\rangle\in (0,\rho)\}$.
Then there exists a neighborhood $V\subset \overline{M}_0$ of $p$ and a continuous function $\pi\colon V\ra V\cap H_v$ such that $V$ is homeomorphic to the mapping cylinder of $\pi$.
\end{cor}

\subsection*{Further observations} Theorem~\ref{T:MAIN_THM}, the main result of this article, shows, if anything, that the set $Z$ cannot be too  pathological in a neighborhood of a nondegenerate K{\L} point. The next result looks at the differential inequality that is, in a certain sense,  opposite to the one in item \eqref{T:KLchara_D} of Theorem~\ref{KLchara}. By opposition to Corollary~\ref{csL} and the classical curve selection lemma  in real analytic geometry (see \cite[Ch.~3]{Mi1}) one could say that the following result is a  ``no curve selection lemma''. 

  \begin{definition} Let $p\in \partial Z$. A relative $C^1$ curve $\gamma\colon([0,\varepsilon], \{0\})\ra (M,Z)$ with $\gamma(0)=p$ is called rectifiable if $\int_0^{\epsilon}|\gamma'(t)|~dt<\infty$.
  \end{definition}
  
  For example, if $\gamma$ is $C^1$ on $(0,\epsilon)$ and there exists $C>0$ such that $|\gamma'(t)|\leq C$ for all $t\in (0,\epsilon)$ then $\gamma$ is rectifiable.
  
  \begin{lem}
  \label{ncsL}
  Let $f\colon (M,Z)\ra ([0,\infty),\{0\})$ be a $C^1$ relative function. Suppose that there exists $\rho>0$ such that
\[ 
	|\nabla_xf|\leq b(f(x)) \text{ for all } x\in f^{-1}(0,\rho)
\]
for a  continuous function $b\colon (0,\rho]\ra (0,\infty)$  satisfying $\int_0^{\rho} \frac{1}{b(t)}~dt=\infty$.

 Let $p\in  \partial Z$. Then there  exists no rectifiable curve $\gamma\colon([0,\varepsilon], \{0\})\ra (M,Z)$ with $\gamma(0)=p$.
\end{lem}


\begin{proof} We will proceed by contradiction. Suppose then that such a curve $\gamma$ exists. Let $\widetilde{\gamma}:=f\circ \gamma$ and note that $\widetilde{\gamma}(0)=0.$
At a point $t>0$ we have:
 \begin{align}\label{eq22}\widetilde{\gamma}'(t)\leq |\widetilde{\gamma}'(t)|=|\langle \nabla_{\gamma(t)}f,\gamma'(t)\rangle|\leq |\nabla_{\gamma(t)}f|\cdot|\gamma'(t)|\leq b(\widetilde{\gamma}(t))\cdot|\gamma'(t)|.
 \end{align}
Fix $t_0\in (0,\varepsilon)$ such that $u_0:=\widetilde{\gamma}(t_0)\in (0,\rho)$.  Let $B(u):=\displaystyle\int_{u_0}^u\frac{1}{b(r)}~dr$ and note that 
\begin{align} 
\label{eq23}
	\lim_{u\searrow 0}B(u)=-\infty.
\end{align}

From \eqref{eq22} we get, for $t<t_0$, that
 \begin{align}
 \label{EQ:NCL_FINAL_INEQ}
 B(\widetilde{\gamma}(t))=\int_{t_0}^t\frac{\tilde{\gamma}'(s)}{b(\widetilde{\gamma}(s))}~ds\geq -\int_{t}^{t_0}|\gamma'(s)|~ds.
 \end{align}
 Since $\displaystyle\lim_{t\searrow 0}\widetilde{\gamma}(t)= 0$,  inequality \eqref{EQ:NCL_FINAL_INEQ}  contradicts \eqref{eq23} because $\gamma$ is rectifiable.
\end{proof}

Under the simple nondegeneracy condition, the functions which satisfy a differential inequality as in Lemma \ref{ncsL} have the following alternative description.

\begin{prop}
\label{eqMZ} 
Let $(M,Z)$ be a relative manifold and let $f\colon(M,Z)\ra ([0,\infty),\{0\})$ be a relative $C^1$ function. Let $p\in \partial Z$ be a simple nondegenerate point.  Then the following assertions are equivalent:
\begin{enumerate}
	\item
	\label{T:eqMZ_ITEM_A}
	 There exists $\rho>0$,  an open neighborhood $U\ni p$, and a continuous function $b\colon (0,\rho]\ra (0,\infty)$ such that 
	 \[
	\int_0^{\rho}\frac{1}{b(r)}~dr=\infty
	\]
	and
	\begin{align}
	\label{Eqnabb}
	|\nabla_xf|\leq b(f(x))\text{ for all } x\in U\setminus Z.
	\end{align}
	\item
	\label{T:eqMZ_ITEM_B}
	 There exists a compact neighborhood $K$ of $p$ such that the finite valued, lower semicontinuous function $\beta^K(t):=\displaystyle\frac{1}{\displaystyle\sup_{x\in f^{-1}(t)\cap K}|\nabla_x f|}$ satisfies $\displaystyle\int_0^{\rho}\beta^K(r)~dr=\infty$ for some $\rho>0$.
\end{enumerate}
\end{prop}
\begin{proof} For $\eqref{T:eqMZ_ITEM_A}\Rightarrow \eqref{T:eqMZ_ITEM_B}$, let $K\subset U$ be a compact neighborhood and note that condition \eqref{Eqnabb} implies that  	$1/\beta^K(t)\leq b(t)$. We get \eqref{T:eqMZ_ITEM_B} by taking the reciprocals and integrating.

To prove $\eqref{T:eqMZ_ITEM_B}\Rightarrow \eqref{T:eqMZ_ITEM_A}$, we use Lemma~\ref{Lem45} to obtain a continuous function $\tilde{b}\colon (0,\rho]\to (0,\infty)$ with   $0<\tilde{b}\leq \beta$ such that $\displaystyle\int_0^\rho\tilde{b}=\infty$. Then $b:=\tilde{b}^{-1}\geq (\beta^K)^{-1}$ will satisfy item \eqref{T:eqMZ_ITEM_A}.
\end{proof}

\begin{example} \label{Exabb} Here is a simple example where the  situation contemplated in Proposition \ref{eqMZ}  occurs. Let $h:([0,\infty),\{0\})\ra (\bR,\{0\})$ be a continuous relative function of class $C^1$ such that the graph $\Gamma_h\subset\bR^2$ is not rectifiable at $\{0\}$, i.e.,  it has infinite length at $0$. Take for example $h(x)=x\sin(x^{-1})$ for $x>0$ and $h(0)=0$. To see that $\int_0^{\epsilon }\sqrt{1+[(h)']^2}=\infty$ one reduces to $\int_0^{\epsilon}x^{-1}|\cos(x^{-1})|~dx=\infty$ and the later follows by bounding below with the sum of integrals over the intervals where $|\cos(x^{-1})|\geq \sqrt{2}/2$, for example. 

Then $(M,Z):=(\Gamma_h, \{(0,0)\})$ is a relative manifold. Let $f$ be the restriction of the projection $(x,y)\ra x$ to $M$. Note that
 \[\qquad |\nabla _{(x,y)}f|=(1+(h'(x))^2)^{-1/2}\neq 0, \qquad\mbox{ for all } (x,y)\in M\setminus Z.\] 
 Define $b(x)=(1+(h'(x))^2)^{-1/2}$ and note that $b(f(x,y))=b(x)$ and hence we have equality in (\ref{Eqnabb}). Clearly $\int_0^{\rho} b^{-1}=\infty$ due to the choice of $h$.
\end{example}

The previous Example should be contrasted to  the following.

\begin{rem} \label{Finrem} If $n=1$ and $0\in \partial f^{-1}(0)$ is a simple nondegenerate point of a continuous  relative function $f\colon([0,\infty),\{0\})\ra ([0,\infty),\{0\})$ which is $C^1$ (on $(0,\infty)$ but not necessarily at $0$),  then $0$ is K{\L} nondegenerate. This follows by noting that $f$ is strictly increasing in a neighborhood of $0$ because $\sgn(f)$ is constant, hence bijective on some $[0,\rho]$. Moreover for $t>0$:
\[\sup_{x\in f^{-1}\{t\}}\frac{1}{|f'(x)|}=\frac{1}{f'(f^{-1}(t))}=(f^{-1})'(t),
\]
which is clearly in $L^1(0,\rho)$ owing to the fact that $f^{-1}:(f([0,\rho]),\{0\})\ra ([0,\rho],0)$ is also a continuous relative function of class $C^1$. 

We also note that for $n=1$, condition $(4')$  \emph{without the simple nondegeneracy condition} holds for every $C^1$ function on a bounded interval as one can learn from  \cite[Theorem 3.2]{Hu}. 
\end{rem}

This section suggests that the simple nondegenerate points $p\in \partial Z$, where $Z$ is the zero locus of a relative $C^1$ function, may be  divided in at most three classes:
\begin{itemize}
\item[(i)] The ``good'' ones, i.e., those that are K{\L} nondegenerate.
\item[(ii)] The ``bad'' ones, i.e., those for which there exists a compact neighborhood $K$ such that
\begin{align*}
 \int_{0}^{\rho} \frac{1}{\sup_{x\in f^{-1}(t)\cap K}|\nabla_xf|}~dt<\infty=\int_{0}^{\rho} \frac{1}{\inf_{x\in f^{-1}(t)\cap K}|\nabla_xf|}~dt
\end{align*}
for some $\rho>0$. The examples in \cite{BDLM} and \cite{BoltePauwels} are not K{\L} but are of this nature as neither  satisfies the conclusion of Lemma \ref{ncsL}. The set $Z$ is a closed convex set of $\bR^2$ while $M=\bR^2$ hence it can be reached from $M\setminus Z$ via rectifiable curves.
\item[(iii)] The ``ugly'' ones, i.e., those for which there exists a compact neighborhood $K$ such that
\begin{align*}
\int_{0}^{\rho} \frac{1}{\sup_{x\in f^{-1}(t)\cap K}|\nabla_xf|}~dt=\infty.
\end{align*}
Examples can be build in every dimension by taking $M=\Gamma_h\times \bR^{n-1}\subset \bR^2\times \bR^{n-1}$ where $h$ is taken as in example \ref{Exabb}, with $f$ the projection onto the first coordinate. 
\end{itemize}


\bibliographystyle{amsplain}
\bibliography{MCN_FINAL}

\end{document}